\documentclass[11pt,reqno]{amsart}

\usepackage{amssymb,amsmath,amsthm,graphicx,xcolor,mathtools,mathrsfs,tabularx,bbm,tikz,url}
\usepackage{mathabx}\changenotsign        
\usepackage{dsfont}
\usepackage[shortlabels]{enumitem}
\usepackage{lmodern}
\usepackage[babel]{microtype}
\usepackage[british]{babel}
\usepackage[utf8]{inputenc}
\allowdisplaybreaks

\usepackage[backref]{hyperref} 
\hypersetup{
	colorlinks,
	linkcolor={red!60!black},
	citecolor={green!60!black},
	urlcolor={blue!60!black}
}

\usepackage{setspace}

\usepackage{geometry}
\let\setminus=\smallsetminus
\let\emptyset=\varnothing

\makeatletter
\def\moverlay{\mathpalette\mov@rlay}
\def\mov@rlay#1#2{\leavevmode\vtop{   \baselineskip\z@skip \lineskiplimit-\maxdimen
		\ialign{\hfil$\m@th#1##$\hfil\cr#2\crcr}}}
\newcommand{\charfusion}[3][\mathord]{
	#1{\ifx#1\mathop\vphantom{#2}\fi
		\mathpalette\mov@rlay{#2\cr#3}
	}
	\ifx#1\mathop\expandafter\displaylimits\fi}
\makeatother

\usepackage{hyperref}
\hypersetup{colorlinks, linkcolor={red!50!black}, citecolor={green!50!black}, urlcolor={blue!50!black}}

\usepackage{caption}
\usepackage{subcaption}

\usepackage[mathlines]{lineno}
\usepackage{etoolbox} 

\newcommand*\linenomathpatch[1]{%
	\expandafter\pretocmd\csname #1\endcsname {\linenomath}{}{}%
	\expandafter\pretocmd\csname #1*\endcsname{\linenomath}{}{}%
	\expandafter\apptocmd\csname end#1\endcsname {\endlinenomath}{}{}%
	\expandafter\apptocmd\csname end#1*\endcsname{\endlinenomath}{}{}%
}
\newcommand*\linenomathpatchAMS[1]{%
	\expandafter\pretocmd\csname #1\endcsname {\linenomathAMS}{}{}%
	\expandafter\pretocmd\csname #1*\endcsname{\linenomathAMS}{}{}%
	\expandafter\apptocmd\csname end#1\endcsname {\endlinenomath}{}{}%
	\expandafter\apptocmd\csname end#1*\endcsname{\endlinenomath}{}{}%
}

\expandafter\ifx\linenomath\linenomathWithnumbers
\let\linenomathAMS\linenomathWithnumbers
\patchcmd\linenomathAMS{\advance\postdisplaypenalty\linenopenalty}{}{}{}
\else
\let\linenomathAMS\linenomathNonumbers
\fi

\linenomathpatchAMS{gather}
\linenomathpatchAMS{multline}
\linenomathpatchAMS{align}
\linenomathpatchAMS{alignat}
\linenomathpatchAMS{flalign}
\linenomathpatch{equation}

\usepackage{cleveref}

\theoremstyle{plain}
\newtheorem{theorem}{Theorem}[section]
\crefname{theorem}{Theorem}{Theorems}

\newtheorem{proposition}[theorem]{Proposition}
\crefname{proposition}{Proposition}{Propositions}

\crefname{corollary}{Corollary}{Corollaries}

\newtheorem{lemma}[theorem]{Lemma}
\crefname{lemma}{Lemma}{Lemmata}

\newtheorem{conjecture}[theorem]{Conjecture}
\crefname{conjecture}{Conjecture}{Conjectures}

\crefname{problem}{Problem}{Problem}

\crefname{claim}{Claim}{Claims}

\crefname{setup}{Setup}{Setups}

\crefname{fact}{Fact}{Facts}

\crefname{algorithm}{Algorithm}{Algorithms}

\crefname{remark}{Remark}{Remarks}

\crefname{example}{Example}{Examples}

\theoremstyle{definition}
\newtheorem{definition}[theorem]{Definition}
\crefname{definition}{Definition}{Definitions}

\newtheorem{observation}[theorem]{Observation}
\crefname{observation}{Observation}{Observations}

\crefname{procedure}{Procedure}{Procedures}

\crefname{construction}{Construction}{Constructions}

\crefname{question}{Question}{Questions}

\numberwithin{equation}{section}

\crefname{section}{Section}{Sections}
\crefname{appendix}{Appendix}{Appendix}

\crefname{figure}{Figure}{Figures}

\theoremstyle{definition}

\newtheorem*{space-prop}{Space}

\def\COMMENT#1{}

\usepackage{comment}

\let\polishlcross=\l
\def\l{\ifmmode\ell\else\polishlcross\fi}
 
\newcommand{\es}{\emptyset}
\newcommand{\eps}{\varepsilon}
\renewcommand{\rho}{\varrho}

\renewcommand{\subset}{\subseteq}

\newcommand{\NATS}{\mathbb{N}}

\newcommand{\REALS}{\mathbb{R}}

\DeclareMathOperator{\Exp}{{E}}

\renewcommand{\phi}{\varphi}

\DeclareMathOperator{\PF}{\mathsf{P}}

\DeclareMathOperator{\lst}{\mathsf{list}}

\DeclarePairedDelimiter{\parens}{(}{)}
\DeclarePairedDelimiter{\set}{\{}{\}}

\DeclarePairedDelimiter{\floor}{\lfloor}{\rfloor}

\DeclarePairedDelimiter{\brackets}{[}{]}
\DeclarePairedDelimiter{\size}{|}{|}

\title{Simultaneous edge-colourings}

\author[S.~Boyadzhiyska]{Simona Boyadzhiyska}
\address{HUN-REN Alfréd Rényi Institute of Mathematics, Budapest, Hungary.}
\email{simona@renyi.hu}

\author[R.~Lang]{Richard Lang}
\address
{
	Departament de Matemàtiques,
	Universitat Politècnica de Catalunya,
	Barcelona, Spain.
}
\email{richard.lang@upc.edu}

\author[A.~Lo]{Allan Lo}
\address{School of Mathematics, University of Birmingham, Birmingham, United Kingdom.}
\email{s.a.lo@bham.ac.uk}

\author[M.~Molloy]{Michael Molloy}
\address{Department of Computer Science, University of Toronto, Canada.}
\email{molloy@cs.toronto.ca}

\thanks{The research leading to these results was supported by EPSRC, grant no. EP/V002279/1 (A.~Lo) and EP/V048287/1 (S.~Boyadzhiyska and A.~Lo).
		There are no additional data beyond that contained within the main manuscript.
		This research was conducted while the first author was at the School of Mathematics, University of Birmingham, Birmingham, United Kingdom.
		\\
	R.~Lang was supported by the Marie Skłodowska-Curie Actions (101018431), the Ramón y Cajal programme (RYC2022-038372-I) and by grant PID2023-147202NB-I00 funded
	by MICIU/AEI/10.13039/501100011033.\\
	M.~Molloy was supported by an NSERC Discovery Grant.}

\begin{document}

	\begin{abstract}
		We study a generalisation of Vizing's theorem, where the goal is to simultaneously colour the edges of graphs $G_1,\dots,G_k$ with few colours. We obtain asymptotically optimal bounds for the required number of colours in terms of the maximum degree $\Delta$, for small values of $k$ and for an infinite sequence of values of $k$. This asymptotically settles a conjecture of Cabello for $k=2$.
		Moreover,  we show that $\sqrt k \Delta + o(\Delta)$ colours always suffice, which tends to the optimal value as $k$ grows.
		We also show that $\ell \Delta + o(\Delta)$ colours are enough when every edge appears in at most $\ell$ of the graphs,  which asymptotically confirms a conjecture of Cambie.
		Finally, our results extend to the list setting.
		We also find a close connection to a conjecture of Füredi, Kahn, and Seymour from the 1990s and an old problem about fractional matchings.
	\end{abstract}
	
	\maketitle
	\vspace{-0.5cm}

	\section{Introduction}\label{sec:intro}
	
	A classic theorem of {Vizing} states that a graph $G$ with maximum degree $\Delta(G) \leq \Delta$ admits a proper edge-colouring using at most $\Delta + 1$ colours.
	We investigate the following natural generalisation of this problem:
	\begin{align*}
		\begin{minipage}[c]{0.9\textwidth}
			{\em Consider graphs $G_1,\dots,G_k$ with $\Delta(G_1),\dots,\Delta(G_k) \leq \Delta$, {possibly sharing some edges.}
				How many colours guarantee an edge-colouring {of} $G_1 \cup \dots \cup G_k$, which is proper for each individual graph $G_1,\dots,G_k$?}
		\end{minipage}\ignorespacesafterend
	\end{align*}
	
	The study of such {simultaneously proper colourings} was proposed by Cabello  and initiated by Bousquet and Durain~\cite{bousquet2020note}. {We} obtain new, asymptotically tight, upper bounds in terms of~$\Delta$. Perhaps surprisingly, the simultaneous colouring question turns out to be closely related to classical problems about colourings and  fractional matchings in hypergraphs.
	
	\medskip
	
	To formalise this discussion, we define an {edge-colouring} $\phi \colon E(G) \to  \NATS$ of a graph $G$ to be \emph{proper} if no two incident edges are assigned the same colour.
	For graphs $G_1,\dots,G_k$, an edge-colouring of $G_1 \cup \dots \cup G_k$ is \emph{simultaneously proper} (or \emph{simultaneous} for short) if its restriction to each individual $G_i$ is proper.
	The \emph{simultaneous chromatic index} $\chi'(G_1,\dots,G_k)$ is the smallest number of colours for which there exists a simultaneous edge-colouring of $G_1,\dots, G_k$.
	Note that for $k=1$ this coincides with the usual chromatic index.
	
	We are interested in bounds on $\chi'(G_1,\dots,G_k)$ in terms of $\Delta$, where $\Delta(G_1),\dots,\Delta(G_k) \leq \Delta$.
	For instance, Vizing's theorem tells us that $\Delta \leq \chi'(G_1) \leq \Delta +1$ when $k=1$.
	Beyond this, only partial results are known.
	
	\medskip
	It is not hard to show that $\chi'(G_1,G_2) \leq \Delta+1$ when $\Delta \leq 2$ (\cref{prop:Delta=2}).
	On the other hand, $\Delta+1$ colours are sometimes required even when $G_1,\dots,G_k$ are bipartite, and hence each $G_i$ satisfies $\chi'(G_i)=\Delta$ (\cref{prop:bipartitelowerbound}).
	Cabello showed that $\chi'(G_1,G_2) \leq \Delta + 2$ when $G_1 \cap G_2$ is regular and conjectured that this bound is  valid for any $G_1$ and $G_2$.
	The first progress in this direction was obtained by Bousquet and Durain~\cite{bousquet2020note}, who proved that
	\begin{align*}
		\chi'(G_1,G_2) \leq 3\Delta/2+4.
	\end{align*}
	We prove that {Cabello's} conjecture is asymptotically true.
	
	For $k\geq 3$, significantly more than $\Delta$ colours may be necessary. Indeed, Bousquet and Durain~\cite{bousquet2020note} provided an example where $G_1\cup G_2\cup G_3$ is a star and  $\chi'(G_1,G_2,G_3)=3\lfloor\frac{\Delta}{2}\rfloor$. We prove a matching upper bound showing that this example is asymptotically extremal.
	
	For general $k$, applying Vizing's theorem to the union  $G_1\cup\dots\cup G_k$ yields $\chi'(G_1,\dots,G_k) \leq k \Delta+1$.
	Bousquet and Durain~\cite{bousquet2020note}  improved this to $2 \sqrt{2k} \Delta - \sqrt{2k} +2 $, which was subsequently sharpened by Cambie~\cite[Section~8.3]{cambie2022extremal} to
	\begin{align*}
		\chi'(G_1,\dots,G_k) \leq \sqrt{2k}\Delta+1.
	\end{align*}
	Turning to lower bounds, Cambie~\cite{cambie2022extremal} modified a construction of Bousquet and Durain~\cite{bousquet2020note} to show a general lower bound, {building collections of graphs requiring approximately $\sqrt{k} \Delta$ colours}. Thus, the upper bound differs from the lower bound by a factor of roughly $\sqrt{2}$. We close this gap by proving an upper bound asymptotic to~$\sqrt{k}\Delta$.
	
	\medskip
	Let us now state our main result.
	It is formulated in terms of a function $\nu(k)\colon \NATS \to \REALS$, which we shall introduce formally and discuss in more detail in \cref{sec:stars}.
	
	\begin{theorem}\label{thm:simultaneous-colourings}
		Let $k\geq 1$ be a fixed integer. Then all graphs $G_1,\dots,G_k$ of maximum degree at most~$\Delta$ satisfy $\chi'(G_1,\dots,G_k) \leq \nu(k)\Delta +o(\Delta)$. Furthermore, {there exist $\Delta$-regular stars $G_1,\dots, G_k$ with the same centre vertex such that}  $\chi'(G_1,\dots,G_k) \geq \nu(k)\Delta -k$.
	\end{theorem}
	
	The function $\nu(k)$ captures the size of a largest fractional matching among all intersecting hypergraphs on $k$ vertices.
	As it turns out, this function is related to a variety of combinatorial problems.
	It was first studied by Mills~\cite{mills1979covering} and later independently by Horák and Sauer~\cite{horak1992covering}, who determined its initial values:
	\begin{align*}
		\centering
		\def\arraystretch{1.5}
		\begin{tabular}{c|ccccccccccccccc}
			$k$      & 1   & 2   & 3             & 4             & 5             & 6   & 7             & 8                & 9             & 10             & 11             & 12 & 13              \\
			\hline
			$\nu(k)$ & $1$ & $1$ & $\frac{3}{2}$ & $\frac{5}{3}$ & $\frac{9}{5}$ & $2$ & $\frac{7}{3}$ & $ \frac{ 12}{5}$ & $\frac{5}{2}$ & $ \frac{8}{3}$ & $\frac{14}{5}$ & 3  & $ \frac{13}{4}$
		\end{tabular}.
	\end{align*}
	
	While it is not hard to compute $\nu(k)$ for small $k$, we do not have an exact  expression for all~$k$. {Indeed, finding such an expression appears to require determining the integers $q$ for which there exists a projective plane of order $q$, and hence may be a very hard problem.}
	However, its asymptotic behaviour is generally understood:
	\begin{equation}\label{eq:nu-bounds}
		\sqrt{k}-o(\sqrt{k})\leq \nu(k) \leq \sqrt k.
	\end{equation}
	A more comprehensive discussion of the function $\nu(k)$ can be found in \cref{sec:stars}.
	For now,  we conclude that these results asymptotically resolve the simultaneous edge-colouring problem and in particular Cabello's conjecture.
	
	\medskip
	
	Next, we turn to a related problem.
	Cambie~\cite{cambie2022extremal} proposed the following natural variant of the simultaneous colouring question in the setting where each edge appears  in only a bounded number of the graphs.
	
	\begin{conjecture}[Cambie]\label{con:cambie}
		Let $G_1,\dots,G_k$ be graphs of maximum degree $\Delta$, where every edge appears at most $\ell$ times.
		Then $\chi'(G_1,\dots,G_k) \leq \rho(\ell) \Delta +o(\Delta)$.
	\end{conjecture}
	
	Note that the conjectured answer is independent of the number of graphs involved.
	For~$\ell=1$, the conjecture reduces to Vizing's theorem, as in this case the graphs are edge-disjoint.
	Moreover, the bounds in \cref{con:cambie} are tight whenever $\ell = \Delta = q+1$ and there is a projective plane of order~$q$~\cite[Section~8.3]{cambie2022extremal}.
	Our work confirms this conjecture asymptotically for all fixed~$k$ and~$\ell$.
	Again, we have an asymptotically tight result, and we state it in terms of a parameter~$\rho(\ell)$ that will be defined precisely in~\cref{sec:stars}.
	
	\begin{theorem}\label{thm:simultaneous-colourings-bounded}
		{Let $k\geq \ell\geq 1$ be fixed integers. All} graphs $G_1,\dots,G_k$ of maximum degree at most $\Delta$, where every edge appears at most $\ell$ times, satisfy $\chi'(G_1,\ldots,G_k) \leq \rho(\ell) \Delta +o( \Delta)$.
		Furthermore, for every $\eps > 0$ and $\Delta_0$, there are $\Delta$-regular stars $G_1,\dots, G_k$ with common centre and $\Delta\geq \Delta_0$ such that every edge appears at most $\ell$ times and $\chi'(G_1,\dots,G_k) \geq \rho(\ell)\Delta - \eps \Delta$.
	\end{theorem}
	
	The function $\rho(\ell)$ captures the best possible bound on the chromatic index of an $\ell$-uniform multihypergraph with maximum degree $\Delta$.  This is a classic problem dating back to at least 1972, when it was posed by Faber and Lov\'asz~\cite{fl}. An old result of Shannon tells us that $\rho(2)=3/2$, but again, determining the exact solution for every $\ell$ appears to be very hard. The following bounds are straightforward and are enough to establish Cambie's conjecture:
	\begin{align}\label{eq:rho-bounds}
		\ell-o(\ell)\leq \rho(\ell) \leq \ell.
	\end{align}
	
	\medskip
	Recall that the extremal examples for both \cref{thm:simultaneous-colourings,thm:simultaneous-colourings-bounded} consist of a collection of stars centred at the same vertex.
	As it turns out, this is not a coincidence.
	Our main technical contribution asymptotically reduces the general problem of simultaneous edge-colouring to the special case of stars.
	For each vertex $v$, we write $R_1(v),\dots, R_k(v)$ for the stars centred at~$v$ in $G_1,\dots, G_k$, respectively.
	By convention, if $v \notin V(G_i)$, then $R_i(v)$ is the trivial star.
	
	\begin{theorem}\label{thm:stars-to-graphs}
		Let $G_1,\dots, G_k$ be graphs of maximum degree at most  $\Delta$ and {$\gamma \geq 0$.} Suppose that, for every vertex $v$, the stars $R_1(v),\dots, R_k(v)$ satisfy $\chi'(R_1(v),\dots, R_k(v))\leq \gamma\Delta$. Then  $\chi'(G_1,\dots, G_k)\leq \gamma\Delta + o(\Delta)$.
	\end{theorem}
	
	Given this, the proofs of \cref{thm:simultaneous-colourings,thm:simultaneous-colourings-bounded} follow from the corresponding statements for stars, which we establish in \cref{sec:stars}.
	
	\medskip
	
	\cref{thm:stars-to-graphs} extends to the more general list colouring setting.
	Since the precise statement is somewhat technical, we defer the details to \cref{sec:main_proof} {(see \cref{thm:stars-to-graphs-list})}.
	Instead, let us highlight the following consequence.
	
	Given a collection of lists $\set{L(e)\colon  e\in G}$ for a graph $G$, an \emph{$L$-edge-colouring} is a colouring $\phi\colon E(G)\to \bigcup_{e\in G} L(e)$ such that $\phi(e) \in L(e)$ for every {$e\in G$}; an $L$-edge-colouring is \emph{proper} if no two incident edges receive the same colour.
	For graphs $G_1,\dots, G_k$ with union $F = G_1 \cup \dots \cup G_k$, we define the \emph{simultaneous list chromatic index} $\chi'_{\lst}(G_1,\dots,G_k)$ as the smallest integer~$r$ such that, for every collection of lists $\set{L(e)\colon  e\in F}$ with $|L(e)| \geq r$ for each~$e \in F$, the union $F$ admits an edge-colouring that is a proper $L$-edge-colouring for each individual~$G_i$.
	
	The notorious List Colouring Conjecture, which emerged during the 1970s and 1980s~\cite{jensen2011graph}, states that $\chi'_{\lst}(G)$ always coincides with $\chi'(G)$. The conjecture  was  confirmed asymptotically in the seminal work of Kahn~\cite{kahn1996asymptotically}.
	We show that the simultaneous colouring problem exhibits a similar asymptotic behaviour.
	
	\begin{theorem}\label{thm:LLC-simultaneous}
		Let $k\geq 1$ be a fixed integer.
		Let $G_1,\dots,G_k$ be graphs of maximum degree at most~$\Delta$.
		Then $\chi'_{\lst}(G_1,\dots,G_k) = \chi'(G_1,\dots,G_k) +o(\Delta)$.
	\end{theorem}
	
	Lastly, we note that the precise analogue of the List Colouring Conjecture fails in the simultaneous setting.
	This follows from a construction adapted from the work of Yannick Mogge and Olaf Parczyk~\cite{pmpersonal} (see \cref{prop:list_counter}).

	\subsection*{Organisation of the paper}
	The rest of the paper is structured as follows.
	In \cref{sec:stars}, we formally introduce, discuss and bound the functions $\nu(k)$ and~$\rho(\ell)$.
	In \cref{sec:main_proof}, we prove the list version of \cref{thm:stars-to-graphs}, as stated in \cref{thm:stars-to-graphs-list}.
	In \cref{sec:constructions}, we provide various constructions for the lower bounds mentioned above.
	Finally, we conclude the paper with a few remarks and open questions in Section~\ref{sec:conclusion}.

	\section{Aiming for the stars}
	\label{sec:stars}
	
	In the following, we present explicit bounds for the simultaneous chromatic index of stars with a common centre.
	It turns out that simultaneously edge colouring stars is equivalent to  edge-colouring  hypergraphs, and consequently is related to a number of previously studied questions.
	{A \emph{hypergraph} $H$ consists of \emph{vertices} $V(H)$ and \emph{edges} $E(H)$, where each edge is a set of vertices.
	We often identify the set of edges with $H$, writing $e \in H$.
	The \emph{degree} of a vertex $v$ is the number of edges containing $v$.
	We write $\Delta(H)$ for the maximum degree (over all vertices) of $H$.
	The concepts of proper edge-colourings (no two intersecting edges receiving the same colour) and the chromatic index generalise straightforwardly to hypergraphs. {In a \emph{multihypergraph}, we allow repeated edges; all other notions carry over in the natural way.}
	
	\begin{definition}[Profile]\label{def:profile}
		Given graphs $G_1,\dots, G_k$, the \emph{profile} of an edge $e \in G_1\cup\dots\cup G_k$ is the set $P(e) = \{ i \in [k]\colon  e \in G_i \}$.
		For a set $S \subset [k]$, we define $G_S \subset G_1\cup\dots\cup G_k$ as the graph of all edges with profile $S$.
		
		The \emph{profile hypergraph} $\PF(G_1,\dots,G_k)$ is the (multi-)hypergraph on vertex set~$[k]$ with an edge $p(e)$ for each $e \in G_1\cup\dots\cup G_k$.
		
	\end{definition}
	
	Note that, if $G_i$ is a star, then the vertex $i$ has degree $\Delta(G_i)$ in the profile hypergraph.
	Furthermore, given a hypergraph~$H$ on $[k]$, one can define stars $G_1,\dots, G_k$ with common centre such that  $P(G_1,\dots, G_k) = H$ and $\Delta(G_i)$ is the degree of vertex~$i$ in~$H$.
	In fact, the following correspondence is immediate from the definition.
	
	\begin{observation}\label{obs:tars-multihypergraphs-index}
		For stars $G_1,\dots,G_k$ with a common centre and $H=\PF(G_1,\dots,G_k)$, we have $\chi'(G_1,\dots,G_k) = \chi'(H)$.
	\end{observation}
	
	We can thus focus on bounding the chromatic index of the profile hypergraph.  Our two main theorems restrict this hypergraph in different ways: \cref{thm:simultaneous-colourings} bounds the number of vertices, and \cref{thm:simultaneous-colourings-bounded} bounds the edge-sizes. We discuss each of these restrictions in the next two subsections.
	
	\subsection{Bounded edge-sizes}
	
	It was shown by Shannon that every multigraph with maximum degree $\Delta$ has chromatic index at most $\lfloor\frac{3}{2}\Delta\rfloor$ and that this bound is tight.  More than 50 years ago, Faber and Lov\'{a}sz~\cite{fl} asked about the best possible bound for \emph{$\ell$-uniform hypergraphs}, namely hypergraphs in which every edge contains exactly $\ell$ vertices.
	Füredi, Kahn and Seymour~\cite{furedi1993fractional} conjectured the following improvement over the trivial bound of roughly $\ell\Delta$:
	
	\begin{conjecture}[Füredi, Kahn and Seymour]\label{con:alon-kim}
		For every fixed integer $\ell$ and every sufficiently large $\Delta$, every $\ell$-uniform multihypergraph $H$ of maximum degree at most $\Delta$ satisfies $\chi'(H) \leq (\ell-1 +1/\ell) \Delta +o(\Delta)$.
	\end{conjecture}
	
	They proved that their conjectured bound holds for the fractional chromatic index.
	For \emph{intersecting} hypergraphs, in which by definition every pair of edges shares at least one vertex, the conjecture follows from the work of F\"uredi~\cite{furedi1993fractional}.
	Alon and Kim~\cite{alon1997degree} posed a more general conjecture and also proved it holds for intersecting hypergraphs.
	
	{\cref{con:alon-kim}}, if true, is asymptotically tight whenever there is  a projective plane of order~$\ell-1$.
	To see this, consider the following construction: Let~$\Delta$ be a large integer divisible by $\ell$ and $k = \ell^2  - \ell +1$.
	Fix a projective plane $P$ of order $\ell-1$ on $k$ vertices and with lines~$e_1,\dots,e_k$.
	We obtain a multihypergraph $H$ by replacing each line $e_i$ with $\Delta/\ell$ copies of itself.
	So~$H$ has maximum degree $\Delta$ and $k\Delta/\ell = (\ell -1 +1/\ell) \Delta$ edges, each of size $\ell$.
	Since $H$ is intersecting, it follows that $\chi'(H) = (\ell -1 +1/\ell)  \Delta$.
	
	We define $\rho(\ell)$ to be the smallest value of $\rho$ such that a bound of $\rho \Delta +o(\Delta)$ holds for this problem.
	This is formalised as follows.
	
	\begin{definition}\label{def:rho}
		For an integer $\ell\geq 2$, let
			\begin{align*}
				\rho(\ell) = \limsup_{\Delta \rightarrow \infty} \left\{ \frac{\chi'(H)}{\Delta (H)} \colon \text{$H$ is an $\ell$-uniform hypergraph with $\Delta(H) \ge \Delta$} \right\}.
			\end{align*}
	\end{definition}
	
	{\cref{con:alon-kim} and the construction following it} posit that $\rho(\ell)\leq \ell -1 +1/\ell$ and  equality holds if there is a projective plane of order $\ell-1$.  This suggests it might be true that $\rho(\ell) = \ell -1 + 1/\ell$ {whenever} there is a projective plane of order $\ell$, in which case determining~$\rho(\ell)$ for every integer~$\ell$ would be a notoriously hard problem.
	
	For general $\ell$, there is still no improvement over the trivial bound $\rho(\ell)\leq \ell$, which follows from the fact that every hyperedge intersects at most $\ell(\Delta-1)$ other hyperedges.
	The above construction, and the density of the values of~$\ell$ for which there is a projective plane of order~$\ell-1$, implies that $\rho(\ell)\geq \ell-o(\ell)$.  This establishes \eqref{eq:rho-bounds}.
	
	With this definition in place, \cref{thm:simultaneous-colourings-bounded} follows immediately from \cref{thm:stars-to-graphs} and \cref{obs:tars-multihypergraphs-index}:
	
	\begin{proof}[Proof of \cref{thm:simultaneous-colourings-bounded}]
		Let $G_1,\dots,G_k$ be a collection of graphs with maximum degree at most~$\Delta$ such that no edge lies in more than $\ell$ of them. Consider any vertex $v$ and let $H(v)$ denote the profile hypergraph of $R_1(v),\dots,R_k(v)$. So by~\cref{obs:tars-multihypergraphs-index}, we have $\chi'(R_1(v),\dots,R_k(v))=\chi'(H(v))$.
		
		As observed below \cref{def:profile}, $H(v)$ has maximum degree at most $\Delta$. Each edge $e\in H(v)$ corresponds to an edge $e'\in R_1(v)\cup\dots\cup R_k(v)$ and the size of $e$ is equal to the number of~$R_i(v)$ containing $e'$; thus $|e|\leq \ell$.  By adding extra degree-one vertices as needed, we can create an $\ell$-uniform hypergraph with the same maximum degree and chromatic index as $H$.
		
		Therefore, by the definition of $\rho(\ell)$, {it follows that} $\chi'(H(v))\leq \rho(\ell)\Delta + o(\Delta)$. \cref{thm:stars-to-graphs} implies that $\chi'(G_1,\dots,G_k)\leq \rho(\ell)\Delta + o(\Delta)$, thus completing the proof.
		
		The lower bound follows from~\cref{obs:tars-multihypergraphs-index} and \cref{def:rho}.
	\end{proof}

	\subsection{Bounded number of vertices}
	
	Next, we turn to the function $\nu(k)$ featured in \cref{thm:simultaneous-colourings}.
	A \emph{fractional matching} of a hypergraph~$H$ is a function $w\colon H \to [0,1]$ such that {$\sum_{e\colon v \in e} w(e) \leq 1$} for every vertex $v \in V(H)$.
	The \emph{size} of $w$ is $\sum_{e \in H} w(e)$, and we write $\nu^\ast(H)$ for the maximum size of a fractional matching in $H$.
	\begin{definition}\label{def:nu}
		The function $\nu(k)$ assumes the maximum of $\nu^\ast(H)$ over all intersecting hypergraphs $H$ on $k$ vertices.
	\end{definition}
	
	{Note that we can assume $H$ to be simple in this definition, as for each set of multiple edges, we can assign the sum of their weights to one of them.}

	The following two observations, proved in \cref{sec:relate-chi-nu}, tie $\nu(k)$ to the chromatic index of multihypergraphs on $k$ vertices.
	Note this is precisely the class of profile hypergraphs that arise in {connection to the simultaneous colouring problem and }\cref{thm:simultaneous-colourings}.
	
	\begin{lemma}\label{lem:index-nu-upper}
		Every multihypergraph $H$ on $k$ vertices has $\chi'(H) \leq \nu(k)\Delta(H)$.
	\end{lemma}
	
	\begin{lemma}\label{lem:index-nu-lower}
		For all positive integers $k$ and $\Delta$, there exists a hypergraph $H$ on $k$ vertices with maximum degree at most $\Delta$ such that $\chi'(H) \geq \nu(k)\Delta -k$.
	\end{lemma}
	
	These two lammata easily yield \cref{thm:simultaneous-colourings}:
	\begin{proof}[Proof of \cref{thm:simultaneous-colourings}]
		This follows straightforwardly from the preceding lammata, \cref{obs:tars-multihypergraphs-index}\ and \cref{thm:stars-to-graphs}, in the same manner as the proof of \cref{thm:simultaneous-colourings-bounded}.  We omit the repetitive details.
	\end{proof}
	
	The function $\nu(k)$ is related to classic work of Lov\'asz~\cite{lovasz2006minimax} and Füredi~\cite{furedi1981maximum}.
	It was first studied by Mills~\cite{mills1979covering} and, independently, {by} Horák and Sauer~\cite{horak1992covering}.
	Since then, it has found many applications in a variety of areas including discrepancy of spanning trees~\cite{gishboliner2022discrepancies}, coverings of edges with cliques~\cite{horak1992covering,mills1979covering}, monochromatic coverings \cite{alon2023power,erdos1990monochromatic} and Ramsey problems~\cite{debiasio2021ramsey}.
	The survey of Füredi~\cite[Chapter 7]{furedi1988matchings} provides further background.
	
	The following result recovers the bounds on $\nu(k)$ given in \eqref{eq:nu-bounds}. It can be viewed as a proof of the analogue to \cref{con:alon-kim}. Those bounds are already known (see for instance~\cite{gishboliner2022discrepancies}), but for the sake of completeness we provide a proof in \cref{sec:bounds-nu}.
	\begin{proposition}\label{prop:upper-bd-nu}
		For all integers $r\geq 1$ and $k\geq r(r+1)$, we have $\nu(k) \leq k/(r+1)$. In particular, $\nu(k)\leq \sqrt{k}$ for all $k\geq 1$.\\ Moreover, the former bound is tight when either
		\begin{enumerate}[label={\rm(\alph*)}]
			\item\label{nu:projective} $k=q^2+q+1$ and there is a projective plane of order $q$; or
			\item\label{nu:affine} $k = q^2+q$ and there is an affine plane of order $q$.
		\end{enumerate}
	\end{proposition}
	
	This proposition suggests that, as with $\rho(\ell)$, determining $\nu(k)$ for every value of $k$ could very well require us to determine the values of $q$ for which there exist projective and affine planes of order $q$.

	\subsection{Chromatic index and fractional matchings}\label{sec:relate-chi-nu}
	
	It remains to show \cref{lem:index-nu-upper,lem:index-nu-lower}.
	
	\begin{proof}[Proof of \cref{lem:index-nu-upper}]
		We can assume that $H$ is intersecting.
		Indeed, as long as there are disjoint edges $e$ and $f$, we can replace them with the single edge~$e \cup f$ to obtain a new graph~$H'$ with the same maximum degree.
		{Then, given any proper edge-colouring of~$H'$, we obtain a proper edge-colouring of~$H$ by assigning the colour of $e\cup f$ to both $e$ and $f$.}
		
		Suppose therefore that $H$ is intersecting and thus $\chi'(H)=\size{E(H)}$.
		Let~$J$ be the (simple) hypergraph on $V(H)$ that contains the edges of positive multiplicity in $H$.
		Define a function  $w\colon J \to [0,1]$ by setting $w(e) = m(e)/\Delta(H)$, where $m(e)$ is the multiplicity of $e$ in $H$.
		It is easy to see that~$w$ is a fractional matching in $J$ {and has} size  $\size{E(H)}/\Delta(H)  \leq \nu^\ast(J) \leq \nu(k)$.
		This then implies that  $\chi'(H)\leq \nu(k)\Delta(H)$, as claimed.
	\end{proof}
	
	For the proof of \cref{lem:index-nu-lower}, we use the following auxiliary result.
	\begin{lemma}[{\cite{furedi1981maximum}}]\label{lem:linear-programming-fact}
		For every hypergraph $H$, there exists a subhypergraph $H' \subset H$ such that $\nu^\ast(H) = \nu^\ast(H')$ and  $\size{E(H')}\leq \size{V(H')}$.
	\end{lemma}
	
	\begin{proof}[Proof of \cref{lem:index-nu-lower}]
		Let $J$ be an intersecting hypergraph on vertex set $[k]$ with a fractional matching $w \colon J \to [0,1]$ of size $\nu(k)$. By \cref{lem:linear-programming-fact}, we can assume that $\size{E(J)} \leq k$.	 We construct the multihypergraph $H$ as follows. For every edge $e\in J$ with {$w(e)>0$}, add $\floor{w(e)\Delta}$ copies of~$e$ as edges in~$H$. Since~$w$ is a fractional matching, the degree of every vertex in $H$ is at most $\Delta$. Since $J$ is intersecting, so is $H$, and therefore $$\chi'(H) = \size{E(H)} = \sum_{e\in J}\floor*{w(e)\Delta}\geq \sum_{e\in J}w(e)\Delta - \size{E(J)}.$$ Combining this with the fact that $J$ was taken to have at most $k$ edges completes the argument.
	\end{proof}

	\subsection{Bounds for fractional matchings}\label{sec:bounds-nu}
	
	For the proof of the upper bound in \cref{prop:upper-bd-nu}, we use an auxiliary lemma.
	
	\begin{lemma}\label{lem:small-edge}
		Let $H$ be an intersecting hypergraph on $k$ vertices with an edge of size at most~$\ell$.
		Then $\nu^\ast(H) \leq \ell$.
	\end{lemma}
	\begin{proof}
		Suppose that $f \in H$ is an edge of size at most $\ell$.
		Consider a fractional matching $w\colon E(H) \to [0,1]$ of size $\nu^\ast(H)$.
		It follows that $$\nu^\ast(H) = \sum_{e \in H} w(f) \leq \sum_{v \in f} \sum_{e: v \in e} w (e) \leq  \sum_{v \in f} 1   = |f| \leq \ell,$$
		where the first inequality follows because $H$ is intersecting and the second inequality because $w$ is a fractional matching.
	\end{proof}
	
	To finish the proof of \cref{prop:upper-bd-nu}, we {consider} the dual to the fractional matching problem (in the linear programming sense).
	For  a hypergraph $H$, a function $t \colon V(H) \to [0,1]$ is a \emph{fractional (vertex) cover} if $\sum_{v \in e} t(v) \geq 1$ for every edge $e \in H$. The \emph{size} of $t$ is $\sum_{v \in V(H)} t(v)$.
	We denote the size of a smallest fractional cover of $H$ by $\tau^\ast(H)$, and by duality we always have $\tau^\ast(H) = \nu^\ast(H)$.
	The \emph{line graph} $L(H)$ of $H$ is a hypergraph  on vertex set~$E(H)$ with edges $\set{e \in E(H)\colon v\in e}$ for all $v\in V(H)$.
	\begin{proof}[Proof of \cref{prop:upper-bd-nu}]
		Consider an intersecting hypergraph $H$ of order~$k$.
		If there is an edge of size at most $r$, then $\nu(k) \leq r \leq k/(r+1)$ by \cref{lem:small-edge}.
		So assume that all edges have size at least $r+1$.
		We define $t \colon V(H) \to [0,1]$ by giving each vertex weight $1/(r+1)$.
		Hence, $t$ is a fractional vertex cover of~$H$ of size $k/(r+1)$.
		So $\nu^\ast(H) = \tau^\ast(H) \leq k/(r+1)$ by linear programming duality.
		
		The tightness comes from a construction based on affine and projective planes.
		
		{\em Part~\ref{nu:projective}:} Suppose $k=q^2+q+1$ and there is a projective plane $P_q$ of order $q$. So $P_q$  is an intersecting $(q+1)$-regular $(q+1)$-uniform hypergraph on $q^2+q+1$ vertices. Assigning weight $ 1/(q+1)$ to every edge of $P_q$ gives a fractional matching of size $k/(q+1)$.
		
		{\em Part~\ref{nu:affine}:} Suppose	 $k=q^2+q$ and
		there is an affine plane $A_q$ of order $q$. So~$A_q$  is a $(q+1)$-regular $q$-uniform hypergraph on $q^2$ vertices with $q^2+q$ edges in which every pair of vertices shares an edge.
		Note that its line graph $L(A_q)$ is an intersecting $q$-regular $(q+1)$-uniform hypergraph on $q^2+q$ vertices with $q^2$ edges. Setting $w(e) = 1/q$ for all $e\in L(A_q)$ then gives a fractional matching of $L(A_q)$ of size $q^2/q = q = k/(q+1)$.
	\end{proof}

	\section{From stars to general graphs}\label{sec:main_proof}
	
	In this section, we prove \cref{thm:stars-to-graphs}, which reduces the problem of simultaneous edge-colourings to the case of stars.
	In fact, we shall show the following stronger version of this result, which extends the phenomenon to list edge-colourings.
	
	\begin{theorem}\label{thm:stars-to-graphs-list}
		For every $\eps>0$ and $k\geq1$, there exists $\Delta_0 = \Delta_0(\eps,k)$ such that the following holds for every  $\Delta \geq \Delta_0$ and $\gamma \geq 0$.
		Let $G_1,\dots, G_k$ be graphs of maximum degree at most $\Delta$.
		Suppose that for all $v\in V(G)$, the stars $R_1(v),\dots, R_k(v)$ satisfy $\chi'(R_1(v),\dots, R_k(v))\leq\gamma\Delta$.
		Then, for every choice of lists~$L(e)$ with $\size{L(e)}\geq (1+2^{k+4}\eps)\gamma\Delta$ for all $e\in G$, the graphs $G_1,\dots, G_k$ have a simultaneously proper $L$-edge-colouring.
	\end{theorem}
	
	{Before giving the proof of \cref{thm:stars-to-graphs-list}, we quickly derive \cref{thm:LLC-simultaneous} from it.}
	\begin{proof}[Proof of \cref{thm:LLC-simultaneous}]
		Let $\eps>0$, $k\geq 1$ and $\Delta\geq \Delta_0$, where $\Delta_0$ is the constant from \cref{thm:stars-to-graphs-list}.
		Let $G_1,\dots, G_k$ be graphs of maximum degree at most~$\Delta$.
		Consider a vertex~$v$  that maximises $\chi'(R_1(v),\dots, R_k(v))$, {and write $\gamma=\chi'(R_1(v), \dots, R_k(v)) / \Delta$.}
		Then  \cref{thm:stars-to-graphs-list} yields
		\begin{align*}
			\gamma\Delta \leq \chi'(G_1,\dots, G_k) \leq \chi'_{\lst}(G_1,\dots,G_k) \leq (\gamma+2^{k+4}\eps)\Delta,
		\end{align*}
		as needed.
	\end{proof}
	
	The rest of this section is devoted to the proof of \cref{thm:stars-to-graphs-list}.
	We begin with an outline of our strategy. Consider $\gamma\geq 0$ as in \cref{thm:stars-to-graphs-list}.  Write $G=G_1\cup\dots\cup G_k$, and let $L(e)$ be a list of size at least $(1+2^{k+4}\eps)\gamma\Delta$ for all~$e\in G$.
	We partition the edges of $G$ into graphs $\{G_S\}_{S \subset [k]}$, where $G_S$  contains the edges with profile~$S$ (cf.\ \cref{def:profile}).
	Our strategy is to colour each $G_S$ independently using a result for local edge-colouring due to Bonamy, Delcourt, Lang and Postle~\cite{bonamy2024edge}. To facilitate this, we sparsify the list of each edge to guarantee that we cannot create a conflict between edges in different~$G_S$, while also maintaining that a proper $L$-edge-colouring exists for each~$G_S$. More precisely, we ensure that, if two {incident} edges appear in the same~$G_i$ but have different profiles, that is, are contained in different $G_S$, then they receive disjoint lists.
	
	We begin with our list sparsification lemma.
	For technical reasons, we write $\delta^+(F)$ to denote the minimum \emph{positive} degree of a graph $F$, that is, the minimum degree of the graph obtained from $F$ after removing any isolated vertices.
	If it is clear from the context, we abbreviate  $\deg_S(v) = \deg_{G_S}(v)$.
	
	\begin{lemma}[List sparsification]\label{lem:list-partitition}
		For every $\eps>0$ and $k\geq1$, there exists $\Delta_0 = \Delta_0(\eps,k)$ such that the following holds for all $\Delta\geq \Delta_0$.
		Let $G_1, \dots, G_k$ be graphs on vertex set $V$ of maximum degree at most $\Delta$, and  $G  = G_1\cup\dots\cup G_k$.
		Suppose that, for each non-empty $S\subseteq [k]$, we have $\delta^+(G_S)\geq \eps\Delta$, and that, for some $1\leq \gamma\leq k$, the stars centred at each~$v\in V$ satisfy  $\chi'(R_1(v),\dots, R_k(v))\leq\gamma\Delta$.
		For each edge $e\in G$, let $L(e)$ be a list of colours satisfying $|L(e)|\geq (1+5\eps)\gamma\Delta$.
		Then there exist subsets $L'(e)\subseteq L(e)$ for each $e\in G$  such that the following holds for every $i \in [k]$ and $S \subset [k]$:
		\begin{enumerate}[label={\rm(\roman*)}]
			\item $L'(xy) \cap L'(xz) = \es$ for all edges $xy,xz\in G_i$ with distinct profiles;
			\label{itm:partition-disjointness}
			\item for every $uv \in G_S$, we have  $$|L'(uv)| \ge (1 + \eps) \frac{\deg_{S}(u)\deg_{S}(v)}{\gamma \Delta };$$
			\label{itm:partition-sizes}
			\item for every vertex $v \in V$  and colour $c \in \bigcup_{uv\in G_S}L'(uv)$, we have
			\label{itm:partition-conflicts}
			\begin{align*}
				\sum_{uv \in G_S \colon c \in L'(uv)} \frac{ \gamma \Delta}{\deg_{S}(u)\deg_{S}(v)} \le 1.
			\end{align*}
		\end{enumerate}
	\end{lemma}
	
	We defer the proof of \cref{lem:list-partitition} to \cref{sec;list-partition} and continue with the proof of \cref{thm:stars-to-graphs-list}.
	Our argument relies on a result due to Bonamy, Delcourt, Lang and Postle~\cite{bonamy2024edge}, which itself is based on the breakthrough work of Kahn~\cite{kahn1996asymptotically}.
	To state the theorem, we introduce a weighted generalisation of list colouring.
	An assignment $(L,\mu)$ of \emph{weighted lists} to a graph $G$ is a collection of lists $L(e)$ and weight functions $\mu(e)\colon L(e)\to (0,1]$ for each~$e\in G$. For simplicity, we write $\mu(e,c) = \mu(e)(c)$.
	
	\begin{theorem}[{\cite[Theorem~6]{bonamy2024edge}}] \label{thm:BDLP}
		For every $\eps>0$, there exists $\delta_0 \ge 1$ such that the following holds for all $\delta \ge \delta_0$.
		Let $G$ be a graph with weighted lists of colours $(L, \mu)$ such that
		\begin{enumerate}[\upshape (a)]
			\item $\mu(e) \colon L(e) \rightarrow [ \delta^{-2}, \delta^{-1} ]$ for every edge $e \in G$;\label{BDLP:mu range}
			\item $\sum_{c \in L(e)} \mu (e,c) \ge 1 + \eps$  for every edge $e \in G$; \label{BDLP:list size}
			\item $\sum_{v\in e} \mu(e, c) \le 1$ for every vertex $ v \in V$ and colour~$c\in \bigcup_{v\in e}L(e)$.\label{BDLP:conflicts}
		\end{enumerate}
		Then there is an $L$-edge-colouring of~$G$.
	\end{theorem}
	
	We are ready now to complete the proof of \cref{thm:stars-to-graphs-list}.
	
	\begin{proof}[Proof of \cref{thm:stars-to-graphs-list}]
		Suppose we are given $k \geq 1$ and $\eps > 0$.
		We may assume that $\eps \leq 2^{-k-10}$.
		Choose $\delta_0$ as in \cref{thm:BDLP} applied with $\eps$.
		Pick $\Delta_0$ as in \cref{lem:list-partitition} applied with $\eps$ and $k$.
		We further assume that {$\Delta_0 \geq k \eps^{-4} \delta_0$.}
		Finally, let $\Delta \geq \Delta_0$ and $\gamma \geq 0$.
		
		Now consider graphs $G_1,\dots, G_k$ of maximum degree at most $\Delta$ on common vertex set $V$.
		We can assume that one of $G_1,\dots,G_k$ has maximum degree $\Delta$, and hence $\gamma \geq 1$.
		{Moreover, since every vertex in $G_1\cup\dots\cup G_k$ has maximum degree at most $k\Delta$, we may assume that $\gamma\leq k$.}
		Suppose we are given a collection of lists $L(e)$ for each $e\in G_1\cup\dots\cup G_k$ satisfying $\size{L(e)}\geq (1+2^{k+4}\eps)\gamma\Delta$.
		Our goal is to prove that  $G_1,\dots, G_k$ can be simultaneously edge-coloured from these lists.
		
		We begin with a preprocessing step, as {\cref{lem:list-partitition}} requires that the graphs $G_S$ have reasonably large minimum degrees. We ensure this by adding some edges to each~$G_S$.
		To this end, let $Q$ be a copy of $K_{2\eps \Delta}$ with a pendent edge. For each $S\subseteq [k]$ and $v \in V$ with $\deg_{S}(v) \le 2 \eps \Delta$, we attach $2 \eps \Delta - \deg_{S}(v)$ many vertex-disjoint copies of $Q$, identifying $v$ with the leaf of $Q$. Denote the resulting graphs by $\hat{G}_S$ for each $S\subseteq[k]$.
		Clearly, we have $2 \eps \Delta \le \delta(\hat{G}_S) \le \Delta(\hat{G}_S) \le \Delta$.
		
		Let $\hat{V} = \bigcup_{S\subseteq [k]}V(\hat{G}_S)$ and $\hat{\Delta} = (1+2^{k+1}\eps)\Delta$.
		Note that,  as a graph on $\hat{V}$, each $\hat{G}_S$ for~$S \subseteq [k]$ satisfies $\delta^+(\hat{G}_S) \ge 2 \eps \Delta  \ge \eps \hat{\Delta} $.
		For each $i\in [k]$, define the graph $\hat{G}_i$ on vertex set~$\hat{V}$ by
		\begin{align*}
			\hat{G}_i = \bigcup_{S \subseteq [k]\colon i \in S} \hat{G}_S .
		\end{align*}
		Note that $\hat{G}_i[V] = G_i$ and that $ \Delta(\hat{G}_i)\leq (1+2^{k+1}\eps)\Delta = \hat{\Delta}$.
		
		Set $\hat{G} = \hat{G}_1\cup\dots\cup \hat{G}_k$.
		By assumption, for each $v\in V$, the stars $R_1(v), \dots, R_k(v)$ have a simultaneously proper colouring using at most $\gamma\Delta$ colours.
		Now, write $\hat{R}_1(v),\dots, \hat{R}_k(v)$ for the stars centred at each $v\in \hat{V}$ in $\hat{G}_1,\dots, \hat{G}_k$. For each $v\in V\subseteq \hat{V}$, we have $\deg_{\hat{G}}(v)\leq \deg_{G}(v)+ 2^{k+1}\eps\Delta$, which implies that the stars $\hat{R}_1(v),\dots, \hat{R}_k(v)$ can be simultaneously coloured with at most $\gamma\Delta+2^{k+1}\eps\Delta \leq \gamma(1+2^{k+1}\eps)\Delta = \gamma\hat{\Delta}$ colours, where we used that $\gamma\geq 1$. Every vertex $v\in \hat{V}\setminus V$ is a non-leaf vertex from a copy of $Q$ and is contained in a single $\hat{G}_S$, thus implying that $\deg_{\hat{G}}(v)\leq 2\eps \Delta\leq \hat{\Delta}$. So the stars around vertices of this kind can also be simultaneously properly coloured with $\gamma\hat{\Delta}$ colours.

		For each edge $e\in G$, we have $\size{L(e)}\ge (1+5\eps) \gamma \hat{\Delta}$.
		We can extend the list assignments by giving each edge in $\hat{G}-(G_1\cup\dots\cup G_k)$ an arbitrary list of size $ \lceil(1+5\eps) \gamma \hat{\Delta}  \rceil$.

		Apply \cref{lem:list-partitition} to $\hat{G}_1,\dots, \hat{G}_k$, with $\hat{\Delta}$ playing the role of $\Delta$, the same~$\gamma$, and  the lists $\set{L(e)\colon e\in \hat{G}}$ to obtain the sparsified lists $\set{L'(e)\colon e\in \hat{G}}$ satisfying properties~\ref{itm:partition-disjointness},~\ref{itm:partition-sizes}, and~\ref{itm:partition-conflicts}. Note that, by property~\ref{itm:partition-disjointness}, it suffices to find a proper $L'$-edge-colouring $\phi_S$ of each $\hat{G}_S$ in turn and then colour the edges of $G$ by setting $\phi(e) = \phi_S(e)$, where $S$ is the profile of $e$ (cf.\ \cref{def:profile}).
		
		For $S \subset [k]$ fixed, our plan is to apply \cref{thm:BDLP} to find a proper $L'$-edge-colouring of~$G_S$.
		In the following, we abbreviate  $\deg_S(v) = \deg_{\hat G_S}(v)$.
		To this end, we define a weighted list assignment by setting, for each edge~$ uv \in \hat{G}_S$ and colour~$c \in L'(uv)$,
		\begin{align*}
			\mu(uv,c) = \frac{\gamma \hat{\Delta}} { \deg_{S}(u) \deg_{S}(v)}.
		\end{align*}
		We verify that this  assignment satisfies the required properties of \cref{thm:BDLP}.
		Let $\delta = \eps^2 \hat{\Delta}/\gamma$.
		Since $\hat{\Delta} \ge \Delta \ge \Delta_0 \geq  k \eps^{-4} \delta_0 \ge \gamma \eps^{-4}\delta_0$, we have $\delta \geq \delta_0 \ge 1$.
		Moreover
		\begin{align*}
			{\delta^2 \geq \frac{\eps^4\Delta^2}{\gamma^2}  \geq \frac{\Delta}{\gamma}\delta_0 \geq \frac{\Delta}{\gamma}.}
		\end{align*}
		For each $uv \in \hat{G}_S$ and colour $c \in L'(uv)$, we have
		\begin{align*}
			\delta^{-2} & \le  \frac{\gamma}{{\Delta}} \le \frac{\gamma \hat{\Delta}} { \deg_{S}(u) \deg_{S}(v)} \le \frac{\gamma}{ \eps^2 \hat{\Delta}}  = \delta^{-1}
		\end{align*}
		verifying property~\ref{BDLP:mu range}.
		Further, for every $uv\in \hat{G}_S$, by property~\ref{itm:partition-sizes}, we have
		\begin{align*}
			\sum_{c \in L'(uv)} \mu (uv,c) & = |L'(uv)| \frac{ \gamma \hat{\Delta}}{\deg_{S}(u)\deg_{S}(v)} \ge  1 + \eps,
		\end{align*}
		so property~\ref{BDLP:list size} is also satisfied.
		Finally, for every $v\in \hat{V}$ and $c\in \bigcup_{uv\in \hat{G}_S}L'(uv)$, property~\ref{itm:partition-conflicts} of the lists $L'$ immediately implies that~\ref{BDLP:conflicts} also holds.
		
		Therefore, \cref{thm:BDLP} yields {a proper $L'$-edge-colouring of each $\hat{G}_S$. Taking the union of these colourings yields}  a simultaneous  colouring of $\hat{G}_1,\dots,\hat{G}_k$ from the lists $L'$, which in turn induces a simultaneous colouring of $G_1,\dots, G_k$.
	\end{proof}

	\subsection{Sparsification}\label{sec;list-partition}
	
	It remains to show \cref{lem:list-partitition}.
	We require the following two probabilistic tools.
	
	\begin{theorem}[Hoeffding's inequality~{\cite{hoeffding1963probability}}]\label{thm:hoeffding}
		Let $X_1,\dots, X_m$ be independent random variables taking values in $[0,c]$, and write $X = X_1+\dots + X_m$ and $\mu = \Exp [X]$. Then, for any $t>0$, we have
		\begin{align*}
			\Pr[\size{X-\mu} \geq t] \leq 2\exp\left({-\frac{2t^2}{ m c^2}}\right).
		\end{align*}
	\end{theorem}
	
	\begin{theorem}[Lov\'asz Local Lemma~{\cite{erdos1975problems}}]\label{thm:LLL}
		Let $A_1,\dots,A_m$ be events in an arbitrary probability space. Suppose that $\Pr[A_i]\leq p$ for all $i\in [m]$ and that each $A_i$ is independent of all but at most $d$ other events. If $(d+1)pe\leq 1$, then with positive probability none of the events occur.
	\end{theorem}
	
	\begin{proof}[Proof of \cref{lem:list-partitition}]
		Without loss of generality, we assume that $\eps<1/10$.
		Moreover, we assume that $\Delta$ is sufficiently large and will specify~$\Delta_0$ in due course. Set $m = (1+\eps) \gamma \Delta$ and $m' = (1+5\eps) \gamma \Delta$.
		We may assume that $\size{L(e)} = m'$ for all $e\in G$.
		
		Recall that, for each vertex $v\in V$ and index $i\in [k]$, we write $R_i(v)$ for the star in $G_i$ consisting of all edges incident to $v$.
		Denote by $R(v)$ the union of the stars $R_1(v),\dots,R_k(v)$.
		By assumption, there is a simultaneous colouring~$\psi_v$ of $R_1(v),\dots,R_k(v)$ using the colours $[m]$.
		
		Write $L(v) = \bigcup_{e\in R(v)}L(e)$.
		Let $\sigma_v\colon L(v) \to [m]$ be a map obtained by picking each~$\sigma_v(c)$ uniformly at random from $[m]$, all choices being made independently.
		For each~$S\subseteq [k]$, write
		\begin{align*}
			L_S(v) = \set{c\in L(v)\colon \sigma_v(c) = \psi_v(e) \text{ for some }e\in G_S \cap R(v)}.
		\end{align*}
		In the following, every edge incident to $v$ in the graph $G_S$ shall take a colour from $L_S(v)$.
		Thus, we assign the available colours for an edge $uv \in G_S$  to be
		\[L'(uv)=L(uv)\cap L_S(u) \cap L_S(v).\]
		Observe that two sets $L_S(v)$ and $L_{S'}(v)$ can intersect only if there exist edges ${vw} \in G_S$ and ${vw'} \in G_{S'}$ with $\psi_v({vw})=\psi_v({vw'})$, which in turn implies that $S \cap S' = \emptyset$.
		In other words, for  any pair of distinct intersecting subsets $S,S'\subseteq [k]$, we have $L_S(v)\cap L_{S'}(v)=\emptyset$.
		
		Consider an arbitrary $i\in [k]$ and a pair of edges $e,e'\in G_i$ that share a vertex~$v$ but have  different profiles $S$ and~$S'$, respectively. Thus $i\in S\cap S'$, so the previous paragraph implies that $L_S(v)\cap L_{S'}(v)=\emptyset$.
		Then
		\begin{align*}
			L'(e) \cap L'(e') \subseteq  L_S(v) \cap  L_{S'}(v) = \emptyset,
		\end{align*}
		so property~\ref{itm:partition-disjointness} always holds.
		
		\medskip
		In the following, we show that properties~\ref{itm:partition-sizes} and~\ref{itm:partition-conflicts} hold with positive probability.
		Observe that, for every colour $c\in L(v)$ and every $S\subseteq [k]$, we have
		\begin{align}\label{eq:probability-c-in-L_S}
			\Pr[c\in L_S(v)] = \frac{\deg_{S}(v)}{m}.
		\end{align}
		Note also that, for fixed $S \subseteq [k]$, the events $c \in L_S(v)$ are {pairwise} independent for vertices~$v $ and colours~$c$. So, for any  $uv\in G_S$, the list size $ \size{L'(uv)}$ follows a binomial distribution with~$m'$ trials and success probability $p = {\deg_{S}(v)\deg_{S}(u)}/{m^2}$.
		Hence
		\begin{align*}
			\Exp [\size{L'(uv)}] = pm' = \frac{1+5\eps}{(1+\eps)^2} \frac{ \deg_{S}(v)\deg_{S}(u)}{\gamma \Delta}
			\ge (1+2 \eps) \frac{ \deg_{S}(v)\deg_{S}(u)}{\gamma \Delta}.
		\end{align*}
		By \cref{thm:hoeffding} applied with $X_1+\dots+X_{m'} \sim \text{Bin}(m',p)$, we have
		\begin{align*}
			& \quad \Pr\left( \size{L'(uv)}  < (1+\eps) \frac{\deg_{S}(u)\deg_{S}(v)}{\gamma\Delta} \right)                             \\
			& \le \Pr\left( \Big| \size{L'(uv)} - \Exp [\size{L'(uv)}] \Big| >  \frac{\eps\deg_{S}(u)\deg_{S}(v)}{\gamma\Delta} \right) \\
			& \leq 2\exp\parens*{-\frac{ 2 (\eps \deg_{S}(v)\deg_{S}(u))^2}{(1+5\eps){\gamma}^3 \Delta^3}}                              \\
			& \le 2\exp\parens*{-\frac{ \eps^2 (\delta^+(G_S))^4}{{\gamma}^3 \Delta^3}}
			\le 2\exp\parens*{-\frac{ \eps^6 \Delta}{\gamma^3}},
		\end{align*}
		{where the final inequality holds since  $\delta^+(G_S)\geq\varepsilon\Delta$. This bounds the probability of~\ref{itm:partition-sizes} failing to hold for a particular $uv\in G_S$.}
		
		Similarly, given $v \in V$, $S \subseteq [k]$,  and $c \in \bigcup_{uv\in G_S}L(uv)$,  write
		
		\begin{align*}
			T(v,S,c) & = \sum_{\substack{uv \in G_S \colon \\  c \in L_S(u)\cap L(uv)}} \frac{\gamma\Delta}{\deg_{S}(u) }
			= 	\sum_{\substack{uv \in G_S \colon  c \in L(uv)}}  \mathbbm{1}_{c \in L_S(u) }\frac{\gamma\Delta}{\deg_{S}(u) },
		\end{align*}
		and note that $T(v,S,c)/\deg_S(v)$ is an upper bound for the quantity we wish to bound in \ref{itm:partition-conflicts}, namely $\sum_{uv \in G_S : c \in L'(uv)} \frac{ \gamma \Delta}{\deg_{S}(u)\deg_{S}(v)}$, since $L'(uv)\subseteq L_S(u)\cap L(uv)$.
		
		{We bound} 
		\begin{align*}
			\Exp \brackets*{T(v,S,c)} & =
			\sum_{\substack{uv \in G_S \colon c \in L(uv)}} \Pr[c\in L_S(u)] \frac{\gamma\Delta }{\deg_{S}(u)}                                    \\
			& = \sum_{\substack{uv \in G_S \colon c \in L(uv)}}\frac{\deg_{S}(u) }{ m} \frac{\gamma\Delta}{\deg_{S}(u)} \\
			& \le  \deg_{S}(v)\frac{\gamma\Delta}{m}
			\le (1-\eps/2)\deg_{S}(v).
		\end{align*}
		
		Our plan is to apply \cref{thm:hoeffding} with a variable $X_u$, defined for each $uv\in G_S$ with $c\in L(uv)$.  We set $X_u$ to be $\frac{\gamma\Delta}{\deg_{S}(u) }$ if $c\in L_S(u)$ and 0 otherwise. Note that $0\leq X_u\leq\frac{\gamma\Delta}{\delta^+(G_S)}\leq \frac{\gamma}{\varepsilon }$, and the number of variables $X_u$ is at most $\deg_S(v)$. This gives
		\begin{align*}
			\Pr \left( T(v,S,c) > \deg_{S}(v) \right)
			& \le
			\Pr \left( \Big|  T(v,S,c) - \Exp \brackets*{T(v,S,c)} \Big| > \frac{\eps}{2} \deg_{S}(v)\right) \\
			& \le
			2 \exp\parens*{-\frac{\varepsilon^2\deg_S(v)^2 }{2\deg_S(v) (\gamma/\varepsilon)^2 }}
			\le 2 \exp\parens*{-\frac{\eps^5 \Delta}{ 2 \gamma^2 }},
		\end{align*}
		where the final inequality holds since  $\deg_S(v)\geq\delta^+(G_S)\geq\varepsilon\Delta$. This bounds the probability of~\ref{itm:partition-conflicts} failing to hold for a particular choice of $v,S,c$.
		
		\medskip
		To finish, observe that each of the above bad events, $\size{L'(uv)}$ being too small for some $uv\in G$ and $T(v,S,c)$ being too large for some choice of $v,S,c$ occurs with probability at most $p' = 2\exp(-\eps^{6}\Delta/\gamma^3)\leq 2\exp(-\eps^{6}\Delta/k^3)$. Note that each of these events is independent of the events associated to vertices or edges at distance at least three from the corresponding vertex or edge. Choose $\Delta_0$ to be such that $4 \exp(-\eps^{6}\Delta_0/k^3)\Delta_0^4\leq 1$. Thus, provided that $\Delta\geq \Delta_0$, we can apply \cref{thm:LLL} with $p'$ as defined above and $d = \Delta^4$ to conclude that properties~\ref{itm:partition-sizes} and~\ref{itm:partition-conflicts} hold with positive probability.
		This shows that there is a deterministic choice for the desired objects.
	\end{proof}

	\section{Lower bound constructions}\label{sec:constructions}
	
	In this section, we provide constructions that shed some light on lower bounds for the simultaneous colouring problem.
	
	First, recall that Cabello conjectured that $\chi'(G_1,G_2)\leq \Delta+2$ for any choice of $G_1,G_2$. Interestingly, it is not {known} whether this is best possible, that is, there are no known pairs requiring $\Delta+2$ colours. It is natural to ask whether $\chi'(G_1,G_2) \le \max\set{\chi'(G_1),\chi'(G_2)}$ always holds.
	It turns out that this is not true.
	
	\begin{proposition} \label{prop:bipartitelowerbound}
		For $\Delta \ge2$, there exist graphs $G_1,G_2$ such that $\Delta(G_i) = \chi'(G_i) = \Delta$ and $G_1 \cup G_2$ is bipartite but $\chi'(G_1,G_2) \ge \Delta+1$.
	\end{proposition}
	
	Recall $\chi'(G) = \Delta$ if $G$ is bipartite by Kőnig's theorem. We show that $\chi'(G_1,G_2)$ can exceed~$\Delta$ even when $G_1 \cup G_2$ is bipartite. We begin with a simple auxiliary proposition, showing that there exists a graph in which every proper $\Delta$-edge-colouring contains some forced structure.
	
	\begin{proposition}\label{prop:same_set}
		Let $\Delta\geq 2$ be an integer and~$H$ be a copy of $K_{\Delta, \Delta}$ with one edge removed. Denote the vertices of degree $\Delta-1$ in $H$ by $x$ and $y$. Then, for any $\Delta$-edge-colouring $\phi$ of~$H$, the set of colours appearing on edges incident to $x$ is the same as that of $y$, that is,
		$\set{\phi(xu) \colon xu\in H} = \set{\phi(yv)\colon yv\in H}$.
	\end{proposition}
	\begin{proof}
		Any proper $\Delta$-edge-colouring of~$H$ corresponds to a decomposition of~$H$ into $\Delta-1$ perfect matchings and one smaller matching. We deduce that both $x$ and $y$ must be isolated vertices in the smaller matching, so they are both missing the same colour.
	\end{proof}
	
	\begin{proof}[Proof of \cref{prop:bipartitelowerbound}]
		Let $H$ be the graph given by~\cref{prop:same_set} with special vertices $\{x,y\}$ and let $a,b$ be two distinct new vertices.
		Define~$G_1$ to be the graph obtained from~$H$ by adding the edges~$ax, by$.
		Define $G_2$ to be such that $E(G_2) = \{ax,xy, yv \colon v \in N_{H}(y) \}$.
		Clearly $G_1 \cup G_2$ is bipartite and $\Delta(G_1) = \Delta(G_2)= \Delta$.
		We now show that $\chi'(G_1,G_2) > \Delta$.
		Suppose to the contrary that there is a simultaneously proper $\Delta$-edge-colouring $\phi$ of~$G_1,G_2$ and write $c = \phi(a x)$.
		By \cref{prop:same_set} applied to~$G_1$, we deduce that $\phi(b y) = \phi(a x) = c$.
		Hence, $c \notin \{\phi(yv) : v \in N_{G_1}(y) \setminus \{b\}  = N_H(y) \}$.
		On the other hand, since $ax,xy \in G_2$, we have $\phi(xy) \ne  \phi(a x) = c$.
		Recall that $y$ has degree~$\Delta$ in~$G_2$, so $\phi(vy) = c$ for some $v  \in N_{G_1} (y) \setminus \{x\} = N_H(y)$, a contradiction.
	\end{proof}
	
	Next, we show that simultaneous analogue of the List Colouring Conjecture is false.
	
	\begin{proposition} \label{prop:list_counter}
		For $\Delta \ge 2$, there exist graphs $G_1,G_2$ such that $\chi'(G_1,G_2) = \Delta$ but $\chi'_{\lst}(G_1,G_2) \ge \Delta+1$.
	\end{proposition}
	
	\begin{proof}
		We again rely on \cref{prop:same_set}.
		Let $H$ and $H'$ be two disjoint copies of the graph given by~\cref{prop:same_set} with special vertices $\{x,y\}$ and $\{x',y'\}$, respectively.
		Let $z$ be a new vertex.
		Define $\hat{G}_1  = H \cup H' \cup \{xz, y x'\}$ and $\hat{G}_2 = H \cup H' \cup \{xz, y y'\}$.
		Let $\hat{G} = \hat{G}_1 \cup \hat{G}_2$.
		Let~$G$ be the graph obtained {by taking} $\Delta$ vertex-disjoint copies of~$\hat{G}$ and identifying the copies of~$z$ into a single vertex~$z_0$.
		Define $G_1$ and $G_2$ analogously.
		
		We now show that $G_1$ and $G_2$ have the desired properties.
		Clearly $\Delta(G_1) =  \Delta(G_2) = \Delta$.
		To see that $\chi'(G_1, G_2) = \Delta$, we first colour all edges containing~$z_0$ and then extend this to a simultaneous colouring of~$G_1, G_2$.
		By \cref{prop:same_set}, within each copy of~$\hat{G}$, the edges $xz$, $yx'$ and $yy'$ have the same colour.
		
		It remains to show that $\chi_{\lst}(G_1,G_2) > \Delta$.
		Define $\{ L(e) \colon e \in G \}$ to be such that
		\begin{align*}
			L(e) = \begin{cases}
				[ \Delta ]                   & \text{if $e \in H$ or $z_0\in e$}, \\
				[\Delta+1, 2 \Delta]         & \text{if $e \in H'$},              \\
				[\Delta-1] \cup \{\Delta+1\} & \text{if $e$ is a copy of $yx'$},  \\
				[\Delta-1] \cup \{\Delta+2\} & \text{if $e$ is a copy of  $yy'$}.
			\end{cases}
		\end{align*}
		Suppose to the contrary that $\phi$ is a simultaneous $L$-edge-colouring for~$G$.
		We consider the copy of~$\hat{G}$ such that $\phi(xz) = \Delta$.
		Then, by the property of~$H$, we have $\phi(yx') = \Delta+1$ and $\phi(yy') = \Delta+2$.
		Note that $H'$ is properly edge-coloured with colours in~$[\Delta+1, 2 \Delta]$, contradicting the property of $x'$ or~$y'$.
	\end{proof}
	
	\section{Concluding remarks}\label{sec:conclusion}
	
	In this paper, we proved new, and asymptotically best possible, bounds on the simultaneous chromatic index $\chi'(G_1,\dots, G_k)$ for any collection of graphs $G_1,\dots, G_k$ with large maximum degrees. We also studied the restricted setting where each edge is allowed to appear in a bounded number of graphs. In both situations, we reduced the {problem of} bounding $\chi'(G_1,\dots, G_k)$ to the case of stars.
	
	The natural open problem is to determine the correct error term.  Cabello used Vizing's theorem to show that if $G_1\cap G_2$ is $\Delta$-regular, then $\chi'(G_1,G_2)\leq \Delta+2$ and conjectured that the bound extends to arbitrary graphs of maximum degree $\Delta$.
	We are unaware of any examples with $\chi'(G_1,G_2) \ge \Delta+2$, or even $\chi'_{\lst}(G_1,G_2) \ge  \Delta+2$.
	So the conjecture could potentially be sharpened to $\Delta+1$.
	In fact, when $\Delta\leq 2$, then $\Delta+1$ colours always suffice {even in the list colouring setting}.
	
	\begin{proposition} \label{prop:Delta=2}
		Let $G_1,G_2$ be graphs of maximum degree $\Delta\leq 2$. Then $\chi'_{\lst}(G_1,G_2)\leq \Delta+1$.
	\end{proposition}
	
	\begin{proof}
		If $\Delta=1$, one colour suffices. Assume now $\Delta=2$ and $\{L(e) \colon e \in G_1 \cup G_2\}$ with $|L(e)| \ge 3$ for all~$e \in G_1 \cup G_2$.
		Note that $G_1\cap G_2$ is a disjoint union of paths and cycles.
		Hence, we fix a proper $L$-edge-colouring~$\phi$ of~$G_1\cap G_2$.
		It remains to colour~$G_1 \setminus G_2$ and~$G_2 \setminus G_1$.
		Note that $G_1$ is a disjoint union of paths and cycles with some edges already properly coloured.
		Since $|L(e)| \ge 3$ for all~$e \in G_1$, we can extend $\phi$ to a proper $L$-edge-colouring of~$G_1$.
		A similar argument also holds for~$G_2$.
	\end{proof}
	
	\cref{prop:list_counter} shows that $\chi'(G_1,G_2)$ does not always coincide with  $\chi'_{\lst}(G_1,G_2)$.
	On the other hand, by \cref{thm:LLC-simultaneous}, the two quantities are asymptotically equal, even in the more general setting of $k$ graphs. It would be interesting to investigate how large the gap between~$\chi'$ and $\chi'_{\lst}$ can be in the simultaneous colouring setting.
	Another question is whether one really needs the (involved) techniques developed to tackle the List Colouring Conjecture to solve problems on simultaneous colouring.  It is conceivable that `simpler' explanations are waiting to be discovered.

	\bibliographystyle{amsplain}
	\bibliography{biblio}
		
\end{document}